\documentclass[12pt,leqno]{article}
\usepackage{amsmath,amssymb,amscd,latexsym}

\usepackage[all]{xy}

\newcommand{\C}{{\mathbb{C}}}
\newcommand{\R}{{\mathbb{R}}}
\newcommand{\Z}{{\mathbb{Z}}}

\newcommand{\ddet}{\mathrm{det}}
\newcommand{\ess}{\mathrm{ess}}

\newcommand{\supp}{\mathrm{supp}\,}

\newcommand{\GL}{\mathrm{GL}\,}
\newcommand{\Gr}{\mathrm{Gr}}

\newcommand{\tr}{\mathrm{tr}}
\newcommand{\Ah}{{\mathcal A}}
\newcommand{\Bh}{\mathcal{B}}

\newcommand{\Hh}{{\mathcal H}}

\newcommand{\Kh}{\mathcal{K}}
\newcommand{\hKh}{\hat{\Kh}}
\newcommand{\oKh}{\overline{\Kh}}
\newcommand{\Lh}{{\mathcal L}}
\newcommand{\Mh}{{\mathcal M}}
\newcommand{\Nh}{{\mathcal N}}

\newcommand{\Rh}{{\mathcal R}}

\newcommand{\Zh}{\mathcal{Z}}

\newcommand{\oc}{\overline{c}}
\newcommand{\hA}{\hat{A}}

\newcommand{\oa}{\overline{a}}

\newcommand{\hE}{\hat{E}}
\newcommand{\hP}{\hat{P}}
\newcommand{\hU}{\hat{U}}
\newcommand{\hgamma}{\hat{\gamma}}
\newcommand{\hOmega}{\hat{\Omega}}
\newcommand{\htau}{\hat{\tau}}
\newcommand{\hPhi}{\hat{\Phi}}
\newcommand{\hRh}{\hat{\Rh}}
\newcommand{\ohne}{\setminus}
\newcommand{\nsum}{\textstyle\sum}
\newcommand{\tei}{\, | \,}

\newcommand{\verk}{\raisebox{0.03cm}{\mbox{\scriptsize $\,\circ\,$}}}
\newtheorem{theorem}{Theorem}

\newtheorem{prop}[theorem]{Proposition}
\newtheorem{cor}[theorem]{Corollary}

\newenvironment{rem}{\noindent {\bf Remark}}{}
\newenvironment{rems}{\noindent {\bf Remarks}}{}
\newenvironment{claim}{\noindent {\bf Claim}\it}{}

\newenvironment{proof}{\noindent {\bf Proof}}{\mbox{}\hspace*{\fill}$\Box$}
\begin{document}
\title{Determinants on von~Neumann algebras, Mahler measures and Ljapunov exponents}
\author{Christopher Deninger}
\date{}
\maketitle
\begin{abstract}
For an ergodic measure preserving action on a probability space, consider the corresponding crossed product von Neumann algebra. We calculate the Fuglede-Kadison determinant for a class of operators in this von Neumann algebra in terms of the Ljapunov exponents of an associated measurable cocycle. The proof is based on recent work of Dykema and Schultz. As an application one obtains formulas for the Fuglede-Kadison determinant of noncommutative polynomials in the von Neumann algebra of the discrete Heisenberg group. These had been previously obtained by Lind and Schmidt via entropy considerations.
\end{abstract}

\section{Introduction} \label{s1}
The Fuglede--Kadison determinant is a very natural generalization of the usual determinant to finite von~Neumann algebras. It was introduced in 1952 in \cite{FK} and has since found several interesting applications. To name but a few this determinant appears in the definition of analytic and combinatorial $L^2$-torsion -- generalizations of Ray--Singer and Reidemeister torsion to non-compact covering spaces \cite{L}. It plays a role in the work of Haagerup and Schultz on the invariant subspace problem for operators in II$_1$-factors \cite{HSch1} and it is closely related to the entropy of algebraic actions of discrete amenable groups \cite{D}, \cite{DS}. 

One of the first and most basic constructions of finite von~Neumann algebras is the group-measure space construction $\Rh$ of Murray and von~Neumann for an automorphism of a finite probability space. It is an obvious problem to calculate the Fuglede--Kadison determinant for the ``polynomial'' operators that generate $\Rh$. This problem could have been stated but probably not solved in 1952. The answer which we give in the present paper requires the Oseledets multiplicative ergodic theorem and recent work of Dykema and Schultz on the Aluthge transform \cite{DSch}. In more detail the situation is the following: For an ergodic automorphism $\gamma$ of a probability space $(\Omega , \Ah , P)$ consider the semidirect product von~Neumann algebra $\Rh = L^{\infty} (\Omega) \rtimes_{\gamma} \Z$. Let $U$ be the unitary operator corresponding to the action of $\gamma$. Then the finite sums of the form $\Phi = \sum_i a_i U^i$ with $a_i$ in $L^{\infty} (\Omega)$ form a strongly dense subalgebra of $\Rh$. In this paper we calculate the Fuglede--Kadison determinant $\det_{\Rh} \Phi$ for a certain class of elements $\Phi$. The answer is given in terms of the Ljapunov exponents $\chi_i$ of a matrix $A_{\Phi}$ over the ring $L^{\infty} (\Omega)$ which is attached to $\Phi$ by a simple receipe. The formula has the form
\begin{equation} \label{eq:1}
\log \ddet_{\Rh} \Phi = \sum_j r_j \chi^+_j \; ,
\end{equation}
where $r_j$ is the multiplicity of $\chi_j$ and $\chi^+ = \max (\chi ,0)$. If there is only one Ljapunov exponent or if all $\chi_j$ are positive the answer is more explicit: we have
\begin{equation} \label{eq:2}
\log \ddet_{\Rh} \Phi = \left( \int_{\Omega} \log |\det A_{\Phi} (\omega)| d P (\omega) \right)^+
\end{equation}
in the first case and
\begin{equation} \label{eq:3}
\log \ddet_{\Rh} \Phi = \int_{\Omega} \log |\det A_{\Phi} (\omega)| \, dP (\omega) 
\end{equation}
in the second. Ljapunov exponents are hard to calculate in general but they have been much studied and the results can be applied to the calculation of Fuglede--Kadison determinants via formula \eqref{eq:1}. If a von~Neumann algebra is the direct integral of semidirect products as above, the logarithmic Fuglede--Kadison determinant being the integral of their logarithmic Fuglede--Kadison determinants can be calculated in terms of \eqref{eq:1} as well. For example, the von~Neumann algebra $\Nh\Gamma$ of the discrete Heisenberg group $\Gamma$ is the direct integral of the rotation algebras $\Rh_{\zeta} = L^{\infty} (S^1) \rtimes_{\zeta} \Z$ for $\zeta \in S^1$. Using formula \eqref{eq:2} we therefore find
\begin{equation} \label{eq:4}
\log \ddet_{\Nh\Gamma} (1 - a (y,z) x) = \int_{S^1} \left( \int_{S^1} \log |a (\eta, \zeta)| \, d\mu (\eta) \right)^+ \, d \mu (\zeta) \; .
\end{equation}
Here $x,y$ are generators of $\Gamma$ with central commutator $z = (y,x)$ and $a (y,z)$ is a polynomial in the commuting elements $y$ and $z$. By $\mu$ we denote the Haar probability measure on the circle $S^1$. For more general noncommutative polynomials in $x,y,z$ there is a similar formula obtained from \eqref{eq:1} which expresses the logarithmic Fuglede--Kadison determinant as an integral over $S^1$ of sums of nonzero Ljapunov exponents, c.f. theorem \ref{t11} below. Formula \eqref{eq:4} and its generalization are originally due to Lind and Schmidt who proved them by two methods which are quite different from the one in the present paper. On the one hand by \cite{DS}, theorem 6.1,  $\det_{\Nh\Gamma}$ can be calculated as a limit of renormalized finite dimensional determinants and the proof of \eqref{eq:4} in \cite{LS} is based on this fact. On the other hand, in \cite{D} and \cite{DS} the entropy $h$ of certain $\Gamma$-actions was expressed in terms of Fuglede--Kadison determinants. Using another method to calculate the entropy, Lind and Schmidt expressed $h$ in terms of the double integral in \eqref{eq:4} and more generally as an integral over Ljapunov exponents. 

The present paper was motivated by the desire to understand the work of Lind and Schmidt from the point of view of von~Neumann algebras. In particular, I wanted to see via functional analysis how the Ljapunov exponents are related to Fuglede--Kadison determinants.

As we show in \S\,\ref{s3} the proof of \eqref{eq:2} and hence of \eqref{eq:4} is not so difficult for continuous functions $a$ if the action is uniquely ergodic as for the irrational rotations that appear in the analysis of the Heisenberg group. Under this assumption one can make use of the uniform convergence in the Birkhoff ergodic theorem. However even unique ergodicity would not help in the proof of \eqref{eq:1} since it does not imply uniform convergence in the Oseledets multiplicative ergodic theorem. Instead we use a method from Margulis' proof of the multiplicative ergodic theorem in \cite{M} to reduce formula \eqref{eq:1} to formula \eqref{eq:2} for measurable functions $a$. Using an automorphism  of $\Rh$ we may assume that $a$ is non-negative. In this case formula \eqref{eq:2} is equivalent to a result on the Brown measure of $aU$ which was proved in \cite{DSch} based on the results of \cite{HSch1}.

I would like to thank Doug Lind and Klaus Schmidt very much for explaining their work to me which motivated the present paper. I am also grateful Siegfried Echterhoff for useful discussions and to the referee for helpful remarks. Part of this research was done at Keio university and the Erwin Schr\"odinger Institute. I am very grateful to these institutions and to my hosts Masanori Katsurada and Klaus Schmidt for their invitations and support.
\section{Determinant and Brown measure} \label{s2}
Let $\Mh$ be a finite von~Neumann algebra with a faithful normal finite trace $\tau$. For an operator $\Phi$ in $\Mh$ let $E_{\lambda} = E_{\lambda} (|\Phi|)$ be the spectral resolution of the selfadjoint operator $|\Phi| = (\Phi^* \Phi)^{1/2}$. Note that $|\Phi|$ and $E_{\lambda}$ lie in $\Mh$. The Fuglede--Kadison determinant $\det_{\Mh} \Phi \ge 0$ of $\Phi$ is defined by the integral:
\begin{equation}\label{eq:5}
\log \ddet_{\Mh} \Phi = \int^{\infty}_0 \log \lambda \, d\tau (E_{\lambda}) \quad \mbox{in} \; \R \cup \{ - \infty \} \; .
\end{equation}
Its definition can be reduced to the case of invertible operators by the formula:
\begin{equation} \label{eq:6}
\ddet_{\Mh} \Phi = \lim_{\varepsilon \to 0+} \ddet_{\Mh} (|\Phi| + \varepsilon) \; .
\end{equation}
It is a nontrivial fact that the determinant is multiplicative:
\begin{equation} \label{eq:7}
\ddet_{\Mh} \Phi_1 \Phi_2 = \ddet_{\Mh} \Phi_1  \; \ddet_{\Mh} \Phi_2 \quad \mbox{for} \; \Phi_1 , \Phi_2 \; \mbox{in} \; \Mh \; .
\end{equation}
The proofs can be found in \cite{FK} \S\,5, \cite{Di} I.6.11 and in greater generality in \cite{B} and \cite{HSch2}. For an easy example consider the abelian von~Neumann algebra $\Zh = L^{\infty} (\Omega)$ attached to a finite measure space $(\Omega , \Ah , \mu)$ equipped with the trace $\tau (\Phi) = \int_{\Omega} \Phi (\omega) \, d\mu (\omega)$. Then we have for any $\Phi$ in $\Zh$:
\begin{equation} \label{eq:8}
\ddet_{\Zh} \Phi = \exp \int_{\Omega} \log |\Phi (\omega)| \, d\mu (\omega) \; .
\end{equation}
In \cite{B} p. 19 Brown proved that for any $\Phi$ in $\Mh$ there is a unique Borel measure $\mu_{\Phi}$ on $\C$ with support contained in the spectrum $\sigma (\Phi)$ of $\Phi$ such that for all $z \in \C$ we have:
\begin{equation} \label{eq:9}
\log \ddet_{\Mh} (z - \Phi) = \int_{\C} \log |z-\lambda| \, d\mu_{\Phi} (\lambda) \; .
\end{equation}
The volume is given by
\begin{equation} \label{eq:10}
\mu_{\Phi} (\C) = \mu_{\Phi} (\sigma (\Phi)) = \tau (1) \; .
\end{equation}
For example, it follows easily from equation \eqref{eq:8} that in the abelian case $\Zh = L^{\infty} (\Omega)$, we have $\mu_{\Phi} = \Phi_* \mu$ if we represent $\Phi$ in $\Zh$ by a measurable map $\Phi : \Omega \to \C$. In general however it is quite difficult to calculate the Brown measure. We refer to \cite{HSch2} and its references for cases where this has been achieved.

In the next proposition, $r_{\sigma}$ denotes the spectral radius.

\begin{prop} \label{t1}
For an operator $\Phi$ in a finite von~Neumann algebra $\Mh$ with a trace $\tau$ as above the following formulas hold:
\begin{equation} \label{eq:11}
\ddet_{\Mh} (z - \Phi) = |z|^{\tau (1)} \quad \mbox{if} \; r_{\sigma} (\Phi) < |z| \; \mbox{and} \; \tau (\Phi^{\nu}) = 0
\end{equation}
for all $\nu \ge 1$. 
\begin{equation} \label{eq:12}
\ddet_{\Mh} (z - \Phi) = \ddet_{\Mh} \Phi \quad \mbox{if $\Phi$ is invertible} \; , \; r_{\sigma} (\Phi^{-1})^{-1} > |z|
\end{equation}
and $\tau (\Phi^{\nu}) = 0$ for all $\nu \le -1$. 
\end{prop}

\begin{proof}
In the situation of \eqref{eq:11}, we have
\begin{equation} \label{eq:13}
\ddet_{\Mh} (z - \Phi) = \ddet_{\Mh} z \; \ddet_{\Mh} (1 - z^{-1} \Phi) = |z|^{\tau (1)} \ddet_{\Mh} (1 - z^{-1} \Phi)
\end{equation}
and $r_{\sigma} (z^{-1} \Phi) < 1$. Set $A = z^{-1} \Phi$. Since $r_{\sigma} (A) < 1$, the series
\[
B = \log (1 -A) = - \sum^{\infty}_{\nu=1} \frac{A^{\nu}}{\nu}
\]
converges in the norm of $\Mh$ and we have $\exp B = 1 - A$ in $\Mh$. According to \cite{Di} Lemma 4, p. 121, we have
\begin{equation} \label{eq:14}
\tau (\log (\exp B^* \exp B)) = \tau (B^*) + \tau (B) \; .
\end{equation}
Now, 
\begin{eqnarray*}
2 \log \ddet_{\Mh} (1 -A) & = & \tau (\log (1 - A^*) (1 -A)) \\
 & = & \tau (\log (\exp B^* \exp B)) \\
 & = & \tau (B^*) + \tau (B) \; .
\end{eqnarray*}
Because of $\tau (\Phi^{\nu}) = 0$ for $\nu \ge 1$ by assumption, we obtain $\tau (B) = 0$ and $\tau (B^*) = \overline{\tau (B)} = 0$. Hence we have $\log \det_{\Mh} (1 -A) = 0$ and formula \eqref{eq:13} implies the assertion. Formula \eqref{eq:12} follows from formula \eqref{eq:11} if we write
\[
\ddet_{\Mh} (z - \Phi) = \ddet_{\Mh} (\Phi) \ddet_{\Mh} (1 - z \Phi^{-1}) \; .
\]
Note that the spectral radius of an invertible operator is positive. 
\end{proof}

\begin{cor} \label{t2}
For $\Mh$ and $\tau$ as above, let $\Phi$ be an invertible operator in $\Mh$ with $\tau (\Phi^{\nu}) = 0$ for all $\nu \neq 0$ and $r_{\sigma} (\Phi^{-1}) = r_{\sigma} (\Phi)^{-1}$. Then we have $\det_{\Mh} \Phi = r_{\sigma} (\Phi)^{\tau (1)}$ and
\begin{equation} \label{eq:15}
\ddet_{\Mh} (z - \Phi) = \max (|z|^{\tau (1)} , \ddet_{\Mh} \Phi) \quad \mbox{if} \; \ddet_{\Mh} \Phi \neq |z|^{\tau (1)} \; .
\end{equation}
The condition $\det_{\Mh} \Phi \neq |z|^{\tau (1)}$ implies that $z - \Phi$ is invertible. The converse holds if the spectrum of $\Phi$ is rotation invariant.
\end{cor}

\begin{proof}
We have $\det_{\Mh} \Phi \le \| \Phi \|^{\tau (1)}$ since $|\Phi|$ has the same norm as $\Phi$. It follows that
\[
\ddet_{\Mh} \Phi = \lim_{n\to \infty} (\ddet_{\Mh} \Phi^n)^{1/n} \le \lim_{n\to \infty} \| \Phi^n \|^{\tau (1) / n} = r_{\sigma} (\Phi)^{\tau (1)} \; .
\]
Applied to $\Phi^{-1}$ we get
\[
\ddet_{\Mh} \Phi^{-1} \le r_{\sigma} (\Phi^{-1})^{\tau (1)} \quad \mbox{hence} \quad \ddet_{\Mh} \Phi \ge r_{\sigma} (\Phi^{-1})^{-\tau(1)} \; .
\]
Since we assumed that $r_{\sigma} (\Phi^{-1}) = r_{\sigma} (\Phi)^{-1}$, the equation $\det_{\Mh} \Phi = r_{\sigma} (\Phi)^{\tau (1)}$ follows. Now formula \eqref{eq:15} is a consequence of proposition \ref{t1}. Because of the formula $r_{\sigma} (\Phi) = \sup |\lambda| , \lambda \in \sigma (\Phi)$ the assumption $r_{\sigma} (\Phi^{-1}) = r_{\sigma} (\Phi)^{-1}$ means that all $\lambda \in \sigma (\Phi)$ have the same absolute value $|\lambda| = r_{\sigma} (\Phi)$. The condition that $r_{\sigma} (\Phi)^{\tau (1)} = \det_{\Mh} \Phi$ be different from $|z|^{\tau (1)}$ therefore implies that $z$ is not in the spectrum of $\Phi$. If on the other hand $\det_{\Mh} \Phi = |z|^{\tau (1)}$, it follows that $|z| = r_{\sigma} (\Phi)$. Assuming that $\sigma (\Phi)$ is rotation invariant we have $\sigma (\Phi) = \{ |\lambda| = r_{\sigma} (\Phi) \}$ because $\sigma (\Phi)$ is non-empty. Hence $z \in \sigma (\Phi)$.
\end{proof}

\section{Semidirect products, automorphisms and determinants} \label{s3}

In this section we construct certain automorphisms of semidirect products and use them to reduce formula \eqref{eq:2} to a result in \cite{DSch}. We also give two elementary proofs of \eqref{eq:2} in the uniquely ergodic case using the results from section \ref{s2}. Useful references are \cite{KR} 2.6.13, 8.6 and 13.1.

Let $\gamma$ be an ergodic measure preserving automorphism of a probability space $(\Omega , \Ah , P)$. Consider the $\C$-vector spaces $\Kh_0 \subset \Kh \subset \oKh$ where $\Kh_0 = \bigoplus_{i \in \Z} L^2 (\Omega) U^i$ and $\oKh = \prod_{i \in \Z} L^2 (\Omega) U^i$ and $\Kh$ the Hilbert space completion of $\Kh_0$ with respect to the scalar product
\[
\left( \nsum_i x_i U^i \; , \; \nsum_j y_j U^j \right) = \nsum_i (x_i , y_i) \; .
\]
Here the $U^i$ are formal variables which serve to distinguish the different copies of $L^2 (\Omega)$. We view $L^2 (\Omega)$ as a sub-Hilbert space of $\Kh$ by mapping $x$ to $x U^0$. Note that $\Kh$ is isomorphic to the Hilbert space $L^2 (\Omega) \hat{\otimes} L^2 (\Z)$. 

A formal series $\Phi = \nsum_{i \in \Z} a_i U^i$ with $a_i$ in $L^{\infty} (\Omega)$ defines a linear map by left multiplication $\Phi : \Kh_0 \to \oKh$. This is done using the rule
\begin{equation} \label{eq:16}
(a U^i) (x U^j) = a (x \verk \gamma^i) U^{i+j}
\end{equation}
for $a \in L^{\infty} (\Omega) , x \in L^2 (\Omega)$ noting that $x \verk \gamma^i$ lies in $L^2 (\Omega)$. Thus we set:
\begin{equation} \label{eq:17}
\Phi \left( \nsum_j x_j U^j \right) = \nsum_n \left( \nsum_{i+j=n} a_i (x_j \verk \gamma^i) \right) U^n \; .
\end{equation}
This is well defined since the inner sums on the right are finite, almost all $x_j$ being zero. Applying $\Phi$ to $U^0 \in \Kh_0$ one sees that the operator \eqref{eq:17} of left multiplication by $\Phi$ determines the formal series for $\Phi$ i.e. the $a_i$'s uniquely. One therefore identifies the formal series with the corresponding operator. The semidirect product $\Rh = L^{\infty} (\Omega) \rtimes_{\gamma} \Z$ is the von~Neumann algebra in $\Bh (\Kh)$ consisting of all such $\Phi$ with $\Phi (\Kh_0) \subset \Kh$ and such that $\Phi : \Kh_0 \to \Kh$ extends to a bounded operator $\Phi$ on $\Kh$. The finite formal series $\Phi$ form a subalgebra of $\Rh$ which is dense in the strong topology of $\Rh$. The von~Neumann algebra $\Rh$ is equipped with the faithful normal finite trace $\tau$ with $\tau (1) = 1$ defined by
\begin{equation} \label{eq:18}
\tau (\Phi) = \int_{\Omega} a_0 (\omega) dP (\omega) = (\Phi U^0 , U^0) \; .
\end{equation}
The algebra $\Rh$ contains the algebra $L^{\infty} (\Omega)$ canonically and $\tau \, |_{L^{\infty} (\Omega)}$ is the integral on $L^{\infty} (\Omega)$. 

In order to define certain unitary operators on $\Kh$ we observe the following facts. For an element $c$ in $L^{\infty} (\Omega)^{\times} = \GL_1 (L^{\infty} (\Omega))$ consider the extension of $c$ to a cocycle:
\begin{equation} \label{eq:19}
\begin{array}{lrcll}
 & c_i (\omega) & = & c (\omega) \cdots c (\gamma^{i-1} (\omega)) & \mbox{for} \; i > 0 \\
& c_i (\omega) & = & c (\gamma^{-1} (\omega))^{-1} \cdots c (\gamma^i (\omega))^{-1} & \mbox{for} \; i < 0 \\
\mbox{and} & c_0 (\omega) & = & 1 \; .
\end{array}
\end{equation}
Then we have the cocycle relations
\begin{equation} \label{eq:20}
c_{i+j} = c_i (c_j \verk \gamma^i) \quad \mbox{for all} \; i,j \in \Z \; .
\end{equation}
Let $L^{\infty} (\Omega , S^1)$ be the subgroup of $L^{\infty} (\Omega)^{\times}$ represented by $S^1$-valued functions. If $c$ is in $L^{\infty} (\Omega , S^1)$ the functions $c_i$ are in $L^{\infty} (\Omega , S^1)$ as well. Hence we can define a unitary operator $Y_c$ on $\Kh$ by setting:
\begin{equation} \label{eq:21}
Y_c \left( \nsum_i x_i U^i \right) = \nsum_i x_i c_i U^i \; .
\end{equation}
It is clear that $Y^*_c = Y_{\oc}$.

The following proposition must be well known but we give the short proof anyhow.

\begin{prop} \label{t3}
For $c$ in $L^{\infty} (\Omega , S^1)$, conjugation with $Y_c$ on $\Bh (\Kh)$ restricts to a
$^*$\nobreakdash-auto\-mor\-phism $\alpha_c$ of $\Rh$. For $\Phi = \nsum_i a_i U^i$ we have
\[
\alpha_c (\Phi) = \nsum_i a_i c_i U^i \; .
\]
Moreover, $\tau (\alpha_c (\Phi)) = \tau (\Phi)$ and $\det_{\Rh} \alpha_c (\Phi) = \det_{\Rh} \Phi$. The Brown measures of $\alpha_c (\Phi)$ and $\Phi$ coincide.
\end{prop}

\begin{proof}
The operator $Y^*_c$ maps $\Kh_0$ to itself. We first show that we have 
\begin{equation} \label{eq:22}
Y_c \verk \Phi \verk Y^*_c = \Phi_c := \nsum_i a_i c_i U^i
\end{equation}
as operators on $\Kh_0$. Because of equation \eqref{eq:17} this follows from the next equations which hold for all $a \in L^{\infty} (\Omega) , x \in L^2 (\Omega) , i,j \in \Z$
\begin{eqnarray} \label{eq:23}
(Y_c \verk aU^i \verk Y_{\oc}) (x U^j) & = & a (x \verk \gamma^i) (\oc_j \verk \gamma^i) c_{i+j} U^{i+j} \\
& \overset{\eqref{eq:20}}{=} & a (x \verk \gamma^i) c_i U^{i+j} \nonumber\\
& = & a c_i (x \verk \gamma^i) U^{i+j} \nonumber\\
& = & (ac_i U^i) (x U^j) \; .\nonumber
\end{eqnarray}
Note that we had to use that $L^{\infty} (\Omega)$ is abelian.

Equation \eqref{eq:22} shows that $\Phi_c (\Kh_0) \subset \Kh$. Moreover $\Phi_c : \Kh_0 \to \Kh$ is bounded since $\Phi$ is bounded and $Y_c$ is unitary. Hence $\Phi_c$ is an element of $\Rh$ and relation \eqref{eq:22} holds  on all of $\Kh$. Hence conjugation by $Y_c$ on $\Bh (\Kh)$ restricts to a $^*$-automorphism $\alpha_c$ on $\Rh$ and we have $\alpha_c (\Phi) = \Phi_c$ for every $\Phi$ in $\Rh$. It is clear that we have $\tau (\Phi_c) = \tau (\Phi)$. Since conjugation by $Y_c$ is an operator norm preserving $^*$\nobreakdash-automorphism of $\Bh (\Kh)$ we get
\[
\tau \log (|\Phi| + \varepsilon) = \tau \alpha_c (\log (|\Phi| + \varepsilon)) = \tau \log (|\alpha_c (\Phi)| + \varepsilon) \; .
\]
The equality $\det_{\Rh} (\Phi) = \det_{\Rh} (\alpha_c (\Phi))$ follows by letting $\varepsilon$ tend to zero and using formula \eqref{eq:6}. Now equality of Brown measures is immediate.
\end{proof}

For polynomials in $U$ with constant coefficients the Fuglede--Kadison determinant can be expressed in terms of Mahler measures.

\begin{prop} \label{t4}
For a polynomial $f (T) = \nsum a_i T^i$ in $\C [T]$ consider the element $\Phi = f (U) = \nsum a_i U^i$ of $\Rh$. Then we have
\begin{equation} \label{eq:24}
\ddet_{\Rh} \Phi = \exp \int_{S^1} \log |f (\zeta)| \, d\mu (\zeta) \; .
\end{equation}
\end{prop}

\begin{proof}
The spectrum of the unitary operator $U$ on $\Kh$ is well known to be equal to $S^1$ because it is non-empty and rotation-invariant: For $\zeta$ in $S^1$ the automorphism $\alpha_{\zeta}$ from proposition \ref{t3} maps $U$ to $\zeta U$. Hence $U -\mu$ is invertible if and only if $\zeta U - \mu$ i.e. $U - \zeta^{-1} \mu$ is invertible. Consider the continuous function $g_{\varepsilon} (\zeta) = \log (|f (\zeta)| + \varepsilon)$ on $S^1$. According to \cite{KR} Theorem 5.2.8 and formula \eqref{eq:18} we have:
\begin{eqnarray*}
\tau g_{\varepsilon} (U) & = & (g_{\varepsilon} (U) U^0 , U^0) \\
 &= & \int_{S^1} g_{\varepsilon} (\zeta) \, d\nu (\zeta) \; .
\end{eqnarray*}
Here $\nu$ is the Borel measure on $S^1$ defined by $\nu (B) = (\chi_B (U) U^0 , U^0) = \tau (\chi_B (U))$ for all Borel sets $B$ of $S^1$. Here $\chi_B$ is the characteristic function of $B$. Thus we have $\nu (S^1) = \tau (1) = 1$ and
\[
\begin{array}{rclcl}
\nu (\zeta^{-1} B) & = & \tau \chi_{\zeta^{-1} B} (U) & = & \tau \chi_B (\zeta U) = \tau \chi_B (\alpha_{\zeta} (U)) = \tau (\alpha_{\zeta} (\chi_B (U))) \\
& = & \tau (\chi_B (U)) & = & \nu (B) \; .
\end{array}
\]
The probability measure $\nu$ is automatically regular and non-trivial i.e. $\nu (U) > 0$ for all non-empty open subsets of $S^1$. Hence $\nu$ is the Haar probability measure $d\mu (\zeta)$ on $S^1$ and we get the formula: 
\begin{equation} \label{eq:25}
\tau g_{\varepsilon} (U) = \int_{S^1} g_{\varepsilon} (\zeta) \, d\mu (\zeta) \; .
\end{equation}
Using \eqref{eq:6}, \eqref{eq:25} and Levi's theorem, we conclude:
\begin{eqnarray*}
\log \ddet_{\Rh} \Phi & = & \lim_{\varepsilon \to 0+} \tau \log (|f (U)| + \varepsilon) = \lim_{\varepsilon \to 0+} \tau g_{\varepsilon} (U) \\
& = & \lim_{\varepsilon \to 0+} \int_{S^1} \log (|f (\zeta)| + \varepsilon) \, d\mu (\zeta) \\
& = & \int_{S^1} \log |f (\zeta)| \, d\mu (\zeta) \; .
\end{eqnarray*}
An alternative proof using the Brown measure is also possible.
\end{proof}

We now give elementary proofs of formula \eqref{eq:2} under somewhat restrictive hypotheses.

\begin{prop} \label{t5}
Let $\Omega$ be a compact metrizable topological space and let $\gamma$ be a uniquely ergodic homeomorphism of $\Omega$ with invariant Borel probability measure $P$. Let $a = a (\omega)$ and $b = b (\omega)$ be nonvanishing continuous functions on $\Omega$. Then we have
\begin{equation} \label{eq:26}
\log \ddet_{\Rh} (1-aU) = \left( \int_{\Omega} \log |a (\omega)| \, dP (\omega) \right)^+
\end{equation}
and more generally
\begin{equation}
  \label{eq:27n}
  \log \ddet_{\Rh} (b-aU) = \max ( \int_{\Omega} \log |a (\omega)| \, dP(\omega) , \int_{\Omega} \log |b(\omega)| \, dP (\omega)) \; .
\end{equation}
\end{prop}

By formula \eqref{eq:8} and the canonical inclusion $L^{\infty} (\Omega) \subset \Rh$ we have:
\[
\log \ddet_{\Rh} (b) = \int_{\Omega} \log |b (\omega)| \, dP (\omega) \; .
\]
Using the multiplicativity of the determinant, formula \eqref{eq:26} applied to $b^{-1} a$ therefore implies formula \eqref{eq:27n}. We give two proofs of \eqref{eq:26}:

{\bf 1.} \begin{proof}
Consider the Brown measure $\mu_{aU}$ of the operator $aU$ in $\Rh$. According to \eqref{eq:9} and \eqref{eq:10} it is a Borel probability measure supported on $\sigma (aU) \subset \C$ and uniquely determined by the formula
\begin{equation} \label{eq:27}
\log \ddet_{\Rh} (z - aU) = \int_{\C} \log |z - \lambda| \, d\mu_{aU} (\lambda) \quad \mbox{for $z$ in $\C$} \; .
\end{equation}
For $\zeta \in S^1$ consider the automorphism $\alpha_{\zeta}$ of $\Rh$. We have
\begin{eqnarray*}
\ddet_{\Rh} (z-aU) & = & \ddet_{\Rh} (\alpha_{\zeta} (z-aU)) = \ddet_{\Rh} (z-a\zeta U) \\
& = & \ddet_{\Rh} (\zeta^{-1} z - aU) \; .
\end{eqnarray*}
Using formula \eqref{eq:27} we see that $\zeta_* \mu_{aU} = \mu_{aU}$ for all $\zeta \in S^1$ i.e. that $\mu_{aU}$ is rotation invariant, c.f. \cite{DSch} proof of theorem 5.4.

For the spectral radius of $aU$ we find
\begin{eqnarray*}
\log r_{\sigma} (aU) & = & \lim_{n\to \infty} \frac{1}{n} \log \| (aU)^n \| \quad \mbox{(operator norm in $\Rh$)} \\
& = & \lim_{n\to \infty} \frac{1}{n} \log \| (aU)^n U^{-n} \|_{\infty} \quad \mbox{(operator norm in $L^{\infty} (\Omega)$)} \\
& = & \lim_{n\to \infty} \frac{1}{n} \log \max_{\omega \in \Omega} |a (\omega) \cdots a (\gamma^{n-1} (\omega))| \\
& = & \lim_{n\to \infty} \max_{\omega \in \Omega} \frac{1}{n} \nsum^{n-1}_{i=0} \log |a (\gamma^i (\omega))| \\
& = & \int_{\Omega} \log |a (\omega)| \, dP (\omega) \; .
\end{eqnarray*}
The last step is valid since by our assumptions the convergence in the Birkhoff ergodic theorem for the continuous function $\log |a (\omega)|$ is {\it uniform}, c.f. \cite{KH} Proposition 4.1.13. A similar calculation shows that
\[
\log r_{\sigma} ((aU)^{-1}) = \int_{\Omega} \log |a(\omega)|^{-1} dP (\omega) = -\log r_{\sigma} (aU) \; .
\]
Hence we have $r_{\sigma} ((aU)^{-1}) = r_{\sigma} (aU)^{-1}$ and therefore the spectrum $\sigma (aU)$ is contained in $|\lambda| = r_{\sigma} (aU) = M (a) > 0$ where
\[
M (a) = \exp \int_{\Omega} \log |a (\omega)| \, dP (\omega) \; .
\]
Applying $\alpha_{\zeta}$ one sees that $\sigma (aU)$ is rotation invariant and hence we have $\sigma (aU) = \{ \lambda \in \C \tei |\lambda| = M (a) \}$. Compare \cite{AP} \S\,4 for very similar arguments.

The Brown probability measure $\mu_{aU}$ is supported in $\sigma (aU)$ and rotation invariant. Hence the support of $\mu_{aU}$ is the circle $\{ |\lambda| = M (a) \}$ and $\mu_{aU}$ is the push-foreward to $\C$ of the unique rotation invariant probability measure on $\{ |\lambda| = M (a) \}$. Hence we have $\mu_{aU} = M (a)_* \mu$ where $\mu$ is the Haar measure on $S^1$ and $M (a) : S^1 \to \C$ maps $\lambda$ to $\lambda M(a)$. Hence we have by \eqref{eq:27}
\begin{eqnarray*}
\log \ddet_{\Rh} (z-aU) & = & \int_{|\lambda| = M (a)} \log |z-\lambda| \, d (M (a)_* \mu) (\lambda) \\
& = & \int_{S^1} \log |z - \lambda M (a)| \, d \mu (\lambda) \\
& = & \log M (a) + \log^+ |zM (a)^{-1}| \quad \mbox{by Jensen's formula} \\
& = & \max (\log |z| , \log M (a)) \; .
\end{eqnarray*}
For $z = 1$ we get formula \eqref{eq:26}. 
\end{proof}

{\bf 2.} \begin{proof}
Since $\tau ((aU)^{\nu}) = 0$ for $\nu \neq 0$ and $r_{\sigma} (aU) = M (a)$ as we have just seen, corollary \ref{t2} gives the formula
\[
\log \ddet_{\Rh} (z - aU) = \max (\log |z| , \log M (a)) \quad \mbox{if} \; |z| \neq M (a) \; .
\]
Both sides of the equation are subharmonic functions of $z$ in $\C$, c.f. \cite{B}. Since they agree on $\C \setminus \{ |z| = M (a) \}$ and since $\{ |z| = M (a) \}$ is a set of Lebesgue measure zero in $\C$ they agree for all $z \in \C$ and in particular for $z = 1$. This follows from \cite{R} formula (7) on p. 344.
\end{proof}

\begin{rem}
In the situation of the proposition the element $z - aU$ is a unit in $\Rh$ if and only if $M (a) \neq |z|$. This follows from corollary \ref{t2} and the preceeding calculations.
\end{rem}

Using a result of \cite{DSch} which in turn is based on the theory developed in \cite{HSch1} we now give a more general version of formula \eqref{eq:2} and some finer information.

\begin{theorem} \label{t6}
Let $(\Omega , \Ah , P)$ be a probability space with an ergodic measure preserving automorphism $\gamma$. Let $a : \Omega \to \C$ be a bounded measurable function which is either non-negative or non-zero $P$-almost everywhere and set
\[
M (a) = \exp \int_{\Omega} \log |a (\omega)| \, dP (\omega) \ge 0 \; .
\]
Then the following assertions hold:\\[0.2cm]
{\bf 1} $\log \det_{\Rh} (1-aU) = \left( \int_{\Omega} \log |a (\omega)| \, dP (\omega) \right)^+$\\[0.2cm]
{\bf 2} We have $\supp \mu_{a U} = \{ |\lambda| = M(a) \}$. Moreover, if $\log |a|$ is integrable i.e. if $M (a) > 0$ the Brown measure of $aU$ is the push-foreward $\mu_{aU} = M (a)_* \mu$ of the Haar measure $\mu$ on $S^1$ via the map $M (a) : S^1 \to \C$ sending $\lambda$ to $\lambda M (a)$. If $\log |a|$ is not integrable, so that $M (a) = 0$ then $\mu_{aU}$ is the Dirac measure supported in $0 \in \C$.\\
{\bf 3} The spectrum $\sigma (aU)$ contains $\supp \mu_{aU} = \{ |\lambda| = M (a) \}$. If $a \in L^{\infty} (\Omega)^{\times}$ then we have $\sigma (aU) = \supp \mu_{aU}$ if and only if the convergence in the Birkoff ergodic theorem for the integrable function $\log |a|$ is uniform in the $\| \; \|_{\infty}$-norm.
\end{theorem}

\begin{proof}
If $a = a (\omega)$ is non-zero $P$-almost everywhere we obtain a function $c$ in $L^{\infty} (\Omega , S^1)$ by setting $c (\omega ) = \frac{a (\omega)}{|a (\omega)|}$. Consider the automorphism $\alpha_{\oc}$ of $\Rh$ defined in proposition \ref{t3}. We have $\alpha_{\oc} (aU) = a \oc U = |a| U$ and hence by proposition \ref{t3}
\[
\ddet_{\Rh} (1 - aU) = \ddet_{\Rh} (1 - |a| U) 
\]
and $\mu_{aU} = \mu_{|a| U}$. It also follows that $\sigma (aU) = \sigma (|a| U)$ since $\lambda - aU$ is invertible if and only if $\alpha_{\oc} (\lambda - aU) = \lambda - |a| U$ is. In \cite{DSch} Theorem 5.4, it is shown that the Brown measures of $U|a|$ and of $M (a) U$ agree for all $a$ in $L^{\infty} (\Omega)$. Since $|a|U = U |a \verk \gamma^{-1}|$ it follows that $\mu_{|a|U}$ equals the Brown measure of $M (a \verk \gamma^{-1})U = M (a) U$. Thus we have $\mu_{aU} = \mu_{M(a) U}$. If $a$ is non-negative we get the same assertions without having to assume that $a$ is non-zero, because we can quote \cite{DSch} theorem 5.4 directly. The Brown measure of $M (a) U$ is rotation invariant as one sees by applying the automorphisms $\alpha_{\zeta}$ for $\zeta \in S^1$. It is supported in the spectrum of $M (a) U$ which by rotation invariance is the set $\{ |\lambda| = M (a) \}$. It follows that we have
\[
\supp \mu_{aU} = \supp \mu_{M (a) U} = \{ |\lambda| = M (a) \} \; .
\]
Thus for $M (a) = 0$ we get $\mu_{aU} = \delta_0$. If $M (a) > 0$ there is only one rotation invariant probability measure on the Borel algebra of $\C$ with support in $\{ |\lambda| = M (a) \}$ namely $M (a)_* \mu$. Thus {\bf 2} is proved. If $M (a) = 0$ we have $\mu_{aU} = \delta_0$ and hence
\begin{eqnarray*}
\log \ddet_{\Rh} (1 - aU) & = & \int_{\C} \log |1-\lambda| (d \delta_0) (\lambda) = \log 1 = 0 \\
& = & \left( \int_{\Omega} \log |a (\omega)| \, dP (\omega) \right)^+  \; ,
\end{eqnarray*}
since the last integral has the value $-\infty$. If $M (a) > 0$ we have $\mu_{aU} = M (a)_* \mu$ and hence
\begin{eqnarray*}
\log \ddet_{\Rh} (1 -aU) & = & \int_{\C} \log |1-\lambda| \, d (M (a)_* \mu) (\lambda) \\
& = & \int_{S^1} \log |1-\lambda M (a)| \, d\mu (\lambda) \\
& = & \log^+ M (a) \quad \mbox{by Jensen's formula} \; .
\end{eqnarray*}
Thus assertion {\bf 1} is proved as well. The inclusion $\sigma (aU) \supset \supp \mu_{aU}$ is a general fact about Brown measures. Now assume that $a \in L^{\infty} (\Omega)^{\times}$. Then $\log |a|$ is essentially bounded and we have
\[
\| (aU)^n \|^{1/n} = \| (aU)^n U^{-n} \|^{1/n}_{\infty} = \| \prod^{n-1}_{i=0} a \verk \gamma^i \|^{1/n}_{\infty} \; .
\]
Hence we get
\begin{eqnarray*}
  r_{\sigma} (aU) & = & \lim_{n\to \infty} \| (aU)^n \|^{1/n} \\
 & = & M (a) \exp (\lim_{n\to \infty} \underset{\omega \in \Omega}{\ess\,\sup} ( \frac{1}{n} \sum^{n-1}_{i=0} \log |a \verk \gamma^i | - \int_{\Omega} \log |a| dP ) ) \; .
\end{eqnarray*}
Thus the equation $r_{\sigma} (aU) = M (a)$ is equivalent to the formula
\[
\lim_{n\to \infty} \underset{\omega \in \Omega}{\ess\,\sup} (\frac{1}{n} \sum^{n-1}_{i=0} \log |a \verk \gamma^i| - \int_{\Omega} \log |a| \, dP ) = 0 \; .
\]
Replacing $a$ by $\oa^{-1}$ we see that the equation $r_{\sigma} (\oa^{-1} U) = M (\oa^{-1})$ is equivalent to:
\[
\lim_{n\to\infty} \underset{\omega \in \Omega}{\ess \inf} (\frac{1}{n} \sum^{n-1}_{i=0} \log |a \verk \gamma^i| - \int_{\Omega} \log |a| \, dP) = 0 \; .
\]
By {\bf 2} we have $\sigma (aU) = \supp \mu_{aU}$ if and only if $\sigma (aU) \subset \{ |\lambda| = M (a) \}$. This in turn is equivalent to the equations $r_{\sigma} (aU) = M (a) = r_{\sigma} ((aU)^{-1})^{-1}$. Now assertion {\bf 3} follows from the equations $M (\oa^{-1})^{-1} = M (a)$ and
\[
r_{\sigma} (\oa^{-1} U) = r_{\sigma} ((\oa^{-1}U)^*) = r_{\sigma} (U^{-1} a^{-1}) = r_{\sigma} ((aU)^{-1}) \; .
\]
\end{proof}

In proposition \ref{t3} we noted that the determinant is invariant with respect to the automorphism $\alpha_c$. In fact, at least if $\Rh$ is a $\mathrm{II}_1$-factor this is true for endomorphisms respecting much less structure. We state the following proposition as a curiosity. It is not needed in the sequel:

\begin{prop} \label{t7}
Let $\Mh$ be a $\mathrm{II}_1$-factor with normalized trace $\tau$ which acts on a separable Hilbert space. Consider a group homomorphism
\[
\alpha : \Mh^{\times} \to \Mh^{\times} \quad \mbox{with} \quad \alpha (\lambda) = \lambda \quad \mbox{for all} \; \lambda \in \R^{\times} \subset \Mh^{\times} \; .
\]
Then we have $\det_{\Mh} (\alpha (\Phi)) = \det_{\Mh} \Phi$ for all $\Phi$ in $\Mh^{\times}$. 
\end{prop}

\begin{proof}
According to \cite{FH} every element $\Psi$ in $\Mh^{\times}$ with $\det_{\Mh} \Psi = 1$ is a finite product of commutators:
\[
\Psi = (\Phi_1 , \Psi_1) \cdots (\Phi_r , \Psi_r) \quad \mbox{with} \; \Phi_i , \Psi_i \in \Mh^{\times} \; .
\]
For $\Phi \in \Mh^{\times}$ set $\lambda = \ddet_{\Mh} \Phi > 0$. Then $\Psi = \lambda^{-1} \Phi \in \Mh^{\times}$ has determinant $1$ and hence we have
\[
\Phi = \lambda (\Phi_1 , \Psi_1) \cdots (\Phi_r , \Psi_r) \; .
\]
By the assumption on $\alpha$ we get
\[
\alpha (\Phi) = \lambda (\alpha (\Phi_1) , \alpha (\Psi_1)) \cdots (\alpha (\Phi_r) , \alpha (\Psi_r)) \; .
\]
Since $\det_{\Mh}$ is trivial on commutators this implies that
\[
\ddet_{\Mh} \alpha (\Phi) = \ddet_{\Mh} \lambda = \lambda = \ddet_{\Mh} \Phi \; .
\]
\end{proof}

\begin{rem}
Any $\C$-algebra homomorphism $\alpha : \Mh\to \Mh$ mapping $1$ to $1$ gives a homomorphism of groups $\alpha : \Mh^{\times} \to \Mh^{\times}$ as above and hence does not affect the determinant. This is not at all clear from the definition of $\det_{\Mh}$. 
\end{rem}

 \section{The multiplicative ergodic theorem and determinants} \label{s4}

 For a matrix $C \in M_N (\C)$ let $\|C\|_{\sigma}$ be the operator norm of $C$ viewed as a map from the Hilbert space $\C^N$ to itself. Thus $\|C \|_{\sigma} = \max \sqrt{\lambda}$ where $\lambda \ge 0$ runs over the eigenvalues of the positive endomorphism $CC^*$. 

 As in section \ref{s3} consider an ergodic measure preserving automorphism $\gamma$ of a probability space $(\Omega , \Ah , P)$. The operator norm on the von~Neumann algebra
 \[
 M_N (L^{\infty} (\Omega)) \subset \Bh (L^2 (\Omega)^N)
 \]
 is given by
 \[ 
 \| A \| = \underset{\omega \in \Omega}{\ess \; \sup} \; \| A (\omega) \|_{\sigma} \; .
 \]
 The semidirect product $M_N (L^{\infty} (\Omega)) \rtimes_{\gamma} \Z$ consists of all formal series $\Phi = \sum A_i U^i$ with $A_i$ in $M_N (L^{\infty} (\Omega))$ such that left multiplication by $\Phi$ defines a bounded operator from $\Kh^N_0$ to $\Kh^N$ and hence an element of $\Bh (\Kh^N)$. We identify $M_N (\Rh)$ with $M_N (L^{\infty} (\Omega)) \rtimes_{\gamma} \Z$. The trace $\tau$ on $\Rh$ is extended to a trace $\tau_N : M_N (\Rh) \to \C$ by setting $\tau_N = \tau \verk \tr_{\Rh}$ where $\tr_{\Rh} : M_N (\Rh) \to \Rh$ is the usual trace of a matrix over a ring. Note that $\tau_N (1) = N$.

 In this section we calculate Fuglede--Kadison determinants of the form $\det_{M_N (\Rh)} (1 -AU)$ for certain $A$ in $M_N (L^{\infty} (\Omega))$. The answer will be given in terms of the Ljapunov exponents of $A$ which we now recall. Consider any measurable map $A : \Omega \to \GL_N (\C)$. It can be used to lift the ergodic automorphism $\gamma$ of $\Omega$ to a measurable automorphism $\Gamma$ of the trivial bundle $\Omega \times \C^N$ by setting $\Gamma (\omega , v) = (\gamma (\omega) , v A (\omega))$. The iterates of $\Gamma$ are given by the formula $\Gamma^n (\omega , v) = (\gamma^n (\omega) , vA_n (\omega))$ for $n \in \Z$. Here $(A_n (\omega))_{n \in \Z}$ is the cocycle attached to $A$: 
 \begin{eqnarray*}
 A_n (\omega) & = & A (\omega) \cdots A (\gamma^{n-1} (\omega)) \quad \mbox{for} \; n > 0 \\
 A_n (\omega) & = & A (\gamma^{-1} (\omega))^{-1} \cdots A (\gamma^n (\omega))^{-1} \quad \mbox{for} \; n < 0
 \end{eqnarray*}
 and $A_0 (\omega) = 1$. Setting $\log^+ x = \max (\log x ,0)$ we have:

 \begin{theorem}[Oseledets] \label{t8}
 Assume that $\log^+ \| A^{\pm 1} (\omega) \|_{\sigma}$ is integrable over $\Omega$. Then there are:\\
 {\bf a} a measurable $\gamma$-invariant subset $\Omega'$ of $\Omega$ with $P (\Omega \ohne \Omega') = 0$\\
 {\bf b} an integer $1 \le M \le N$ and real numbers $\chi_1 < \chi_2 < \ldots < \chi_M$, the Ljapunov exponents of $A$\\
 {\bf c} positive integers $r_1 , \ldots , r_M$ with $r_1 + \ldots + r_M = N$, the multiplicities of the $\chi_j$\\
 {\bf d} measurable maps $V_j : \Omega' \to \Gr_{r_j} (\C^N)$ into the Grassmannian space of $r_j$-dimensional subspaces of $\C^N$,\\
 such that the following assertions hold for all $\omega \in \Omega'$\\
 {\bf i} $\C^N = \bigoplus^M_{j=1} V_j (\omega)$\\
 {\bf ii} $V_j (\omega) A (\omega) = V_j (\gamma (\omega))$ and hence $V_j (\omega) A_n (\omega) = V_j (\gamma^n (\omega))$ for all $n \in \Z , 1 \le j \le M$.\\
 {\bf iii} For $v \in V_j (\omega) , v \neq 0$ we have
 \[
 \lim_{n\to \pm \infty} \frac{1}{n} \log \frac{\| v A_n (\omega)\|}{\| v \|} = \chi_j \quad \mbox{uniformly in} \; v \; .
 \]
 {\bf iv} $\sum^M_{j=1} r_j \chi_j = \int_{\Omega} \log | \det A (\omega)| \, dP (\omega)$.
 \end{theorem}

 For the proof see \cite{O}, \cite{M} V (2.1) Theorem and remark (5) and App.~A. If $A$ is constant, the Ljapunov exponents are the logarithms of the absolute values of the eigenvalues of $A$. For $N = 1$ and $A$ given by a scalar function $a : \Omega \to \C^*$ the only Ljapunov exponent is: 
 \[
 \chi = \int_{\Omega} \log |a(\omega)| \, dP (\omega) \; .
 \]
 This follows from {\bf iv} or from Birkhoff's ergodic theorem applied to $\log |a (\omega)|$. One relation between determinants in semidirect products and Ljapunov exponents is given by the following theorem.

 \begin{theorem} \label{t9}
 Let $\gamma$ be a measure preserving ergodic automorphism of a probability space $(\Omega , \Ah , P)$. Assume that $(\Omega , \Ah)$ is a standard Borel space. Let $A : \Omega \to \GL_N (\C)$ be a measurable map for which $\|A (\omega) \|_{\sigma}$ is essentially bounded and $\log^+ \|A^{-1} (\omega) \|_{\sigma}$ integrable over $\Omega$. Viewing $A$ as an element of $M_N (L^{\infty} (\Omega)) \subset M_N (\Rh)$ we have the following formula:
 \begin{equation} \label{eq:28}
 \log \ddet_{M_N (\Rh)} (1 - AU) = \sum^M_{j=1} r_j \chi^+_j \; .
 \end{equation}
 Set $M_j = \exp \chi_j$ and let $\mu$ be Haar measure on $S^1$. The Brown measure of $AU$ is given by
 \[
 \mu_{AU} = \nsum^M_{i=1} r_j (M_{j^*} \mu) \; .
 \]
 Its support is the union of the circles $\{ |\lambda| = M_j \}$ for $1 \le j \le M$. If $M = 1$, so that there is only one Ljapunov exponent, we have
 \begin{equation} \label{eq:29}
 \log \ddet_{M_N (\Rh)} (1 - AU) = \left( \int_{\Omega} \log |\det A (\omega)| \, dP(\omega) \right)^+ \; .
 \end{equation}
 On the other hand, if all Ljapunov exponents are non-negative we have:
 \begin{equation} \label{eq:30}
 \log \ddet_{M_N (\Rh)} (1 - AU) = \int_{\Omega} \log |\det A (\omega)| \, dP (\omega) \; .
 \end{equation}
 \end{theorem}

 \begin{proof}
 Formulas \eqref{eq:29} and \eqref{eq:30} follow from \eqref{eq:28} using assertion {\bf iv} of the Oseledets theorem. As for the Brown measure, note that for $z \in \C^*$ using \eqref{eq:28} we have:
 \begin{eqnarray}
 \log \ddet_{M_N (\Rh)} (z - AU)  & = & \log |z|^N + \log \ddet_{M_N (\Rh)} (1 - z^{-1} AU) \nonumber \\
  & = & \nsum^M_{i=1}  r_i \max (\log |z| , \chi_i) \; . \label{eq:31} 
 \end{eqnarray}
 Since both sides are subharmonic \cite{B} this identity must hold for $z = 0$ as well, c.f. \cite{R} formula (7) on p. 344. Alternatively, this follows from an easy extension of equation \eqref{eq:8}:
 \begin{eqnarray*}
 \log \ddet_{M_N (\Rh)} (AU) & = & \log \ddet_{M_N (L^{\infty} (\Omega))} (A) \\
 & = & \int_{\Omega} \log |\det A (\omega)| \, dP (\omega) = \nsum^M_{i=1} r_i \chi_i \; .
 \end{eqnarray*}
 Now, combining \eqref{eq:31} with the equation
 \[
 \max (\log |z| , \chi) = \int_{\C} \log |z - \lambda| \, d (\exp \chi)_* \mu (\lambda) \; ,
 \]
 the formula for the Brown measure of the operator $AU$ in $M_N (\Rh)$ follows.

 It remains to prove formula \eqref{eq:28}. Essentially following the proof of Oseldets' theorem in \cite{M} App.~A we reduce to the case where $A$ is in triangular form. Then the assertion will be deduced from theorem \ref{t6}. 

 Let $B \subset \GL_N (\C)$ be the subgroup of upper triangular matrices and consider the projective variety of right cosets $D = B \ohne \GL_N (\C)$ with its Borel algebra $\Bh (D)$. We equip $\hat{\Omega} = \Omega \times D$ with the product $\sigma$-algebra $\hat{\Ah} = \Ah \otimes \Bh (D)$. The automorphism $\gamma$ of $\Omega$ extends to a measurable automorphism $\hat{\gamma}$ of $\hat{\Omega}$ by setting $\hat{\gamma} (\omega , d) = (\gamma (\omega) , d A (\omega))$. Note that $\GL_N (\C)$ acts on $D$ by right multiplication. Then we have
 \[
 \hat{\gamma}^n (\omega , d) = (\gamma^n (\omega) , d \, A_n (\omega)) \quad \mbox{for all} \; n \in \Z \; .
 \]
 Let $\pi : \hat{\Omega} \to \Omega$ be the projection. Since $(\Omega , \Ah)$ is a standard Borel space the argument in \cite{M} App.~A2 shows that there is a probability measure $\hat{P}$ on $\hat{\Omega}$ with $\pi_* \hat{P} = P$ such that $\hat{\gamma}$ acts as a measure preserving ergodic automorphism on $(\hat{\Omega} , \hat{\Ah} , \hat{P})$. Set $\hat{A} = A \verk \pi : \hat{\Omega} \to \GL_N (\C)$. Then $\hat{A}$ is essentially bounded and $\log^+ \| \hat{A}^{-1} (\;)\|_{\sigma} = (\log^+ \|A^{-1} (\;) \|_{\sigma}) \verk \pi$ is integrable over $\hat{\Omega}$. By setting $\hat{V}_j (\hat{\omega}) = V_j (\pi (\hat{\omega})), \hat{\Omega}' = \pi^{-1} (\Omega') = \Omega' \times D$ we get a Ljapunov decomposition for $\hat{A}$ over $\hat{\Omega}'$ from the one for $A$ over $\Omega'$. 

 The Ljapunov exponents and multiplicities are the same for $A$ and $\hat{A}$. By \cite{FG} Theorem 1 there exists a Borel section $f : D \to \GL_N (\C)$ of the natural projection $\GL_N (\C) \to D$ such that $f (D)$ is relatively compact in $\GL_N (\C)$. For the map $\hat{f} : \hat{\Omega} \to \GL_N (\C)$ defined by $\hat{f} (\omega , d) = f (d)$ the functions $\| \hat{f} (\hat{\omega}) \|_{\sigma}$ and $\| \hat{f} (\hat{\omega})^{-1} \|_{\sigma}$ are therefore bounded for $\hat{\omega} \in \hat{\Omega}$. Hence $\hat{f}$ defines an invertible element of $M_N (L^{\infty} (\hat{\Omega})) \subset M_N (\hat{\Rh})$ where $\hat{\Rh} = L^{\infty} (\hat{\Omega}) \rtimes_{\hat{\gamma}} \Z$. Denoting by $\hat{U}$ the canonical unitary operator in $\hat{\Rh}$ we therefore have the formula:
 \begin{equation} \label{eq:32}
 \ddet_{M_N (\hat{\Rh})} (1 - \hat{A} \hat{U}) = \ddet_{M_N (\hat{\Rh})} (1 - \hat{f} \hat{A} \hat{U} \hat{f}^{-1}) \; .
 \end{equation}
 Now, by definition of $f$ we have for all $g \in \GL_N (\C)$ that $f (Bg) = b_g g$ for some $b_g \in B$. This implies that the matrix $f (d) A (\omega) f (dA(\omega))^{-1}$ is in $B$ for every $d \in D$. Therefore $E (\hat{\omega}) = \hat{f} (\hat{\omega}) \hat{A} (\hat{\omega}) \hat{f} (\hat{\gamma} (\hat{\omega}))^{-1}$ is in $B$ as well for every $\hat{\omega} \in \hat{\Omega}$. Since $\hat{U} \hat{f}^{-1} \hat{U}^{-1} = (\hat{f} \verk \hat{\gamma})^{-1}$ it follows that $\hat{f} \hat{A} \hat{U} \hat{f}^{-1} = E \hat{U}$ in $M_N (\hat{\Rh})$. Here we view the essentially bounded function $E : \hat{\Omega} \to B \subset \GL_N (\C)$ as an element of $M_N (L^{\infty} (\hat{\Omega})) \subset M_N (\hat{\Rh})$. Using formula \eqref{eq:32} we now get the equation
 \begin{equation} \label{eq:33}
 \ddet_{M_N (\hat{\Rh})} (1 - \hat{A} \hat{U}) = \ddet_{M_N (\hat{\Rh})} (1 - E \hat{U}) \; .
 \end{equation}
 Since $1 - E \hat{U} \in M_N (\hat{\Rh})$ is a triangular matrix over $\hat{\Rh}$, it follows from \cite{B} 1.8 Proposition that we have:
 \begin{equation} \label{eq:34}
 \ddet_{M_N (\hat{\Rh})} (1 - E \hat{U}) = \prod^N_{i=1} \ddet_{\hat{\Rh}} (1 - e_{ii} \hat{U}) \; .
 \end{equation}
 Here $e_{11} , \ldots , e_{NN} \in L^{\infty} (\hat{\Omega})$ are the diagonal entries of $E$. They are non-zero in every point of $\hat{\Omega}$, and $\log |e_{ii}|$ is in $L^1 (\hat{\Omega})$ for every $1 \le i \le N$ since $\log^+ \| \hat{A}^{\pm 1} (\_) \|_{\sigma}$ is in $L^1 (\hat{\Omega})$. The Ljapunov exponents of $\hat{A}$ i.e. of $A$ and their multiplicities are equal to the ones of $E$ since $\hat{A}$ and $E$ are $\Lh$-cohomologous c.f. \cite{M} A.1.

 \begin{claim}
 The Ljapunov exponents of the triangular matrix $E$ are the numbers $\int_{\hat{\Omega}} \log |e_{ii}| \, d\hat{P}$ counted with their multiplicities.
 \end{claim}

 {\bf Proof of the claim}
 According to \cite{M} A.3 (*) the matrix $E$ is $\Lh$-cohomologous by a triangular Ljapunov function with $1$'s in the diagonal to a triangular matrix $E_0$ of class $\Psi_0$. But for the matrices in $\Psi_0$ the Ljapunov exponents are the integrals over the logs of the absolute values of the diagonal entries by loc.~cit. and the Birkhoff ergodic theorem. Since the Ljapunov exponents of $E$ and $E_0$ agree and also their diagonal entries the claim follows.

 Thus we have seen that the Ljapunov exponents $\chi_j$ of $A$ are the numbers $\int_{\hOmega} \log |e_{ii}| \, d\hP$ with their multiplicities. 

 Applying theorem \ref{t6} gives the relation:
 \begin{equation} \label{eq:35}
 \log \ddet_{\hRh} (1 - e_{ii} \hU) = \left( \int_{\hOmega} \log |e_{ii}| \, d\hP \right)^+ \; .
 \end{equation}
 Together with equations \eqref{eq:33} and \eqref{eq:34} we finally get:
 \begin{equation} \label{eq:36}
 \log \ddet_{M_N (\hat{\Rh})} (1 - \hA \hU) = {\nsum^M_{j=1}} r_j \chi^+_j \; .
 \end{equation}
 Together with the next claim for $\Phi = 1 - AU$, formula \eqref{eq:28} and hence the theorem follow.

 \begin{claim}
 Let $I \subset \Z$ be a finite subset and $A_i \in M_N (L^{\infty} (\Omega))$ for $i \in I$. Consider the elements $\Phi = \sum_{i \in I} A_i U^i$ in $M_N (\Rh)$ and $\hat{\Phi} = \sum_{i\in I} \hA_i \hU^i$ in $M_N (\hat{\Rh})$ where $\hA_i = A_i \verk \pi$. Then we have $\det_{M_N (\Rh)} (\Phi) = \det_{M_N (\hat{\Rh})} (\hat{\Phi})$.
 \end{claim}

{\bf Proof of the claim}
 Consider the pullback operators
 \[
 \pi^* : L^2 (\Omega) \to L^2 (\hat{\Omega}) \quad \mbox{and} \quad \gamma^* : L^2 (\Omega) \to L^2 (\Omega) \; \mbox{and} \; \hat{\gamma}^* : L^2 (\hat{\Omega}) \to L^2 (\hat{\Omega}) \; .
 \]
 Since we have $\pi_* \hP = P$ the map $\pi^*$ is an isometric embedding which we will view as an inclusion. The operators $\gamma^*$ and $\hgamma^*$ are unitary. Let $\Hh$ be the orthogonal complement of $L^2 (\Omega)$ in $L^2 (\hOmega)$. The relation $\pi \verk \hgamma = \gamma \verk \pi$ implies that $\pi^* \verk \gamma^* = \hgamma^* \verk \pi^*$ i.e. that the operator $\hgamma^*$ on $L^2 (\hOmega)$ restricts to the operator $\gamma^*$ on $L^2 (\Omega)$. We will view $L^{\infty} (\Omega)$ via the norm preserving inclusion $\pi^* : L^{\infty} (\Omega) \to L^{\infty} (\hOmega)$ as a subalgebra of $L^{\infty} (\hOmega)$. Multiplication by elements $a \in L^{\infty} (\Omega)$ respects the subspace $L^2 (\Omega)$ in $L^2 (\hOmega)$. Recall the Hilbert space $\Kh = \hat{\bigoplus_{i\in \Z}} L^2 (\Omega) U^i$ and define the Hilbert space $\hKh = \hat{\bigoplus_{i \in \Z}} L^2 (\hOmega) \hU^i$ similarly. We will view $\Kh$ in the obvious way as a subspace of $\hKh$. Left multiplication by $\hPhi$ on $\hKh^N$ respects $\Kh^N$ and $\hPhi$ restricted to $\Kh^N$ is given by left multiplication with $\Phi$. It follows that the operators $|\hPhi|$ and $\hE_{\lambda} = E_{\lambda} (|\hPhi|)$ in $M_N (\hRh)$ respect $\Kh^N$ and on $\Kh^N$ restrict to the operators $|\Phi|$ and $E_{\lambda} = E_{\lambda} (|\Phi|)$ in $M_N (\Rh)$. Noting that we identified $U^0 \in \Kh$ with $\hU^0 \in \hKh$ we therefore get:
 \begin{eqnarray}
 \htau_N (\hE_{\lambda}) & = & (\hE_{\lambda} (\hU^0) , \hU^0)  \label{eq:37}\\
  & = & (E_{\lambda} (U^0) , U^0) = \tau_N (E_{\lambda}) \nonumber
 \end{eqnarray}
 Therefore we find:
 \begin{eqnarray*}
 \log \ddet_{M_N (\Rh)} (\Phi) & = & \int^{\infty}_0 \log \lambda \;  d\tau_N (E_{\lambda}) \\
 & = & \int^{\infty}_0 \log \lambda \; d\htau_N (\hE_{\lambda}) = \log \ddet_{M_N (\hRh)} (\hat{\Phi}) \; .
 \end{eqnarray*}
 Thus the claim is proved and theorem \ref{t9} as well.
 \end{proof}

 We can now prove formula \eqref{eq:1} of the introduction. Consider an element $\Phi$ of $\Rh$ of the form
 \begin{equation} \label{eq:38}
 \Phi = a_N U^N + \ldots + a_1 U + 1 \quad \mbox{with} \; a_i \in L^{\infty} (\Omega)
 \end{equation}
 and let $A_{\Phi}$ be the matrix
 \begin{equation} \label{eq:39}
 A_{\Phi} = \left(
 \vcenter{\xymatrix@=10pt{
 0 \ar@{}[ddd] |{\mbox{\Large 0}} & 1 &  & \ar@{}[d] |{\mbox{\Large 0}}\\
   & 0 \ar@{.}[dr]& 1 \ar@{.}[dr] & \\
   &   & 0 & 1 \\
 -a_N \ar@{.}[rrr] & & & -a_1
 }} \right) \quad \mbox{in} \; M_N  (L^{\infty} (\Omega)) \; .
 \end{equation}
 Then we have the following result:

 \begin{theorem} \label{t10}
 Let $\gamma$ be a measure preserving ergodic automorphism of a probability space $(\Omega , \Ah , P)$ where $(\Omega , \Ah)$ is a standard Borel space. Assume that $\log |a_N|$ is integrable. Then $A_{\Phi}$ satisfies the assumptions of theorems \ref{t8} and \ref{t9} and we have the formula
 \begin{equation} \label{eq:40}
 \log \ddet_{\Rh} \Phi = {\nsum_j} r_j \chi^+_j \; .
 \end{equation}
 Here the $\chi_j$'s and $r_j$'s are the Ljapunov exponents of $A_{\Phi}$ and their multiplicities.
 \end{theorem}

 \begin{rem}
 Note that we have
 \[
 {\nsum_j} r_j \chi_j = \int_{\Omega} \log |\det A_{\Phi} (\omega)| \, dP (\omega) = \int_{\Omega} \log |a_N (\omega)| \, dP (\omega) \; .
 \]
 Thus equation \eqref{eq:40} implies formulas \eqref{eq:2} and \eqref{eq:3} from the introduction.
 \end{rem}

 \begin{proof}
 Since $\log |a_N|$ is integrable we may represent $a_N$ by a non-vanishing measurable function. We have $|\det A_{\Phi} (\omega)| = |a_N (\omega)|$. Hence we can represent $A_{\Phi}$ by a measurable map $A_{\Phi} : \Omega \to \GL_N (\C)$ for which $\| A_{\Phi} (\omega) \|_{\sigma}$ is bounded and $\log^+ \|A^{-1}_{\Phi} (\omega) \|_{\sigma}$ is integrable over $\Omega$. According to equation \eqref{eq:28} we have
 \[
 \log \ddet_{M_N (\Rh)} (1 - A_{\Phi} U) = \sum_j r_j \chi^+_j \; .
 \]
 It therefore remains to show the equality:
 \begin{equation} \label{eq:41}
 \ddet_{\Rh} \Phi = \ddet_{M_N (\Rh)} (1 - A_{\Phi} U) \; .
 \end{equation}
 Successive left multiplication for $i = 1 , \ldots , N-1$ of the matrix $U^* - A_{\Phi}$ in $M_N (\Rh)$ by elementary matrices of the form
 \[
 E_i = \left(
 \vcenter{\xymatrix@=10pt{
 1 \ar@{.}[drrr] & & 0 & \\
 0 \ar@{.}[r] & e_i \ar@{.}[r] & 0 & 1
 }} \right) \quad \mbox{in} \; \GL_N (\Rh)
 \]
 with $e_i = - (a_N U^i + \ldots + a_{N - i+1} U)$ at $i$-th place transforms $U^* - A_{\Phi}$ into the matrix
 \[
 \left(
 \vcenter{\xymatrix@=10pt{
 U^* \ar@{.}[dr] & -1 \ar@{.}[dr] & \\
  & U^* & -1 \\
 0 \ar@{.}[r] & 0 & \Phi U^* 
 }}
 \right) \; .
 \]
 By \cite{B} 1.8 Proposition the Fuglede--Kadison determinant of a triangular matrix over $\Rh$ is the product of the determinants of the diagonal entries. Hence multiplication by $E_i$ does not effect the determinant and we find
 \[
 \ddet_{M_N (\Rh)} (1 - A_{\Phi} U) = \ddet_{M_N (\Rh)} (U^* - A_{\Phi}) = \ddet_{\Rh} \Phi \; .
 \]
 \end{proof}

 \begin{rem}
 For $\Psi = U^N + c_{N-1} U^{N-1} + \ldots + c_0$ with $c_0 , \ldots , c_{N-1}$ in $L^{\infty} (\Omega)$, we have $\ddet_{\Rh} \Psi = \ddet_{\Rh} \Phi$ where $\Phi = a_N U^N + \ldots + a_1 U + 1$ and $a_i = \oc_{N-i} \verk \gamma^i$. This follows from the relation $\Phi = U^N \Psi^*$. Theorem \ref{t10} can be applied to calculate $\ddet_{\Rh} \Psi$ if $\log |c_0|$ is integrable.
 \end{rem}

\section{The Fuglede--Kadison determinant for the discrete Heisenberg group} \label{sec5}
Let $(Z , \Bh , \mu)$ be a probability space and let $\Mh = \int_Z \Mh (\zeta) \, d\mu (\zeta)$ be a decomposable von Neumann algebra. Assume that we have a measurable field $\zeta \mapsto \tau_{\zeta}$ of faithful normal and finite traces on the $\Mh (\zeta)$'s such that $\tau_{\zeta} (1) = 1$ for all $\zeta \in Z$. Then according to \cite{Di} II.Ch.5, Ex. 2 the integral $\tau = \int_Z \tau_{\zeta} \, d\mu (\zeta)$ defines a faithful normal finite trace on $\Mh$ with $\tau (1) = 1$. For the corresponding determinants we have the formula:
\begin{equation} \label{eq:42}
\log \ddet_{\Mh} \Phi = \int_Z \log \ddet_{\Mh_{\zeta}} \Phi_{\zeta} \, d\mu (\zeta) 
\end{equation}
for every element $\Phi = \int_Z \Phi_{\zeta} \, d\mu (\zeta)$ in $\Mh$. This fact is asserted in \cite{Di} II.Ch.5, Ex. 2 for invertible $\Phi$, in which case $\Phi_{\zeta}$ is invertible for $\mu$-almost all $\zeta$ as well. The general case follows with formula \eqref{eq:6} and Levi's theorem.

Let $\Gamma$ be the integral Heisenberg group with generators $x,y$ and commutator $y^{-1} x^{-1} yx = z$ in the center of $\Gamma$. The Pontrjagin dual of the center is the circle $S^1$. For $\zeta \in S^1$ let $\Rh_{\zeta}$ be the semidirect product $\Rh_{\zeta} = L^{\infty} (S^1) \rtimes_{\zeta} \Z$ corresponding to the rotation by $\zeta$ on $S^1$ with its Haar measure $\mu$. It is known that the von~Neumann algebra $\Nh\Gamma$ of $\Gamma$ decomposes in the form $\Nh\Gamma = \int_{S^1} \Rh_{\zeta} \, d\mu (\zeta)$. The canonical trace on $\Rh_{\zeta}$ will be denoted by $\tau_{\zeta}$. The integral $\tau = \int_{S^1} \tau_{\zeta} \, d\mu (\zeta)$ gives the standard trace on the group algebra $\Nh\Gamma$. The operators $x$ and $y$ decompose as follows: $x = \int_{S^1} U_{\zeta} \, d\mu (\zeta)$ where $U_{\zeta}$ is the operator in $\Rh_{\zeta}$ which was simply denoted $U$ in the previous sections. Moreover $y$ decomposes into the multiplication operators by the coordinate of the $\zeta$'th copy of $S^1$ c.f. \cite{AP} for the $C^*$-algebra case. The complex group ring $\C\Gamma$ of $\Gamma$ is a subalgebra of $\Nh\Gamma$. For complex polynomials $a_i (y,z)$ in the commuting variables $y$ and $z$ we can form the non-commutative polynomial
\[
\Phi = {\nsum^N_{i=0}} a_i (y,z) x^i \quad \mbox{in} \; \C\Gamma \subset \Nh\Gamma \; .
\]
It decomposes into the operators
\[
\Phi_{\zeta} = {\nsum^N_{i=0}} a_i (\_ , \zeta) U^i_{\zeta} \quad \mbox{in} \; \Rh_{\zeta} \; .
\]
By a classical result of Weyl, multiplication by $\zeta \in S^1$ acts uniquely ergodically on $S^1$ if $\zeta$ is not a root of unity. In this case $\Rh_{\zeta}$ is a $\mathrm{II}_1$-factor. Theorems \ref{t6} and \ref{t10} give the following result. For the proof note that $\log |a|$ is integrable over $S^1 \times S^1$ for any two-variable polynomial $a = a (\eta , \zeta)$ which is not the zero polynomial. In particular for $\mu$-almost all $\zeta \in S^1$ we then have $\log |a (\; , \zeta)| \in L^1 (S^1)$ so that $a (\; , \zeta)$ is non-zero on $S^1$ $\mu$-almost everywhere. 

\begin{theorem} \label{t11}
a) For any element of the form $\Phi = 1 - a (y,z) x$ in $\C\Gamma$ we have
\begin{equation} \label{eq:43}
\log \ddet_{\Nh\Gamma} \Phi = \int_{S^1} \left( \int_{S^1} \log |a (\eta, \zeta)| \, d\mu (\eta) \right)^+ \, d\mu (\zeta) \; .
\end{equation}
b) For $\Phi = a_N (y,z) x^N + \ldots + a_1 (y,z) x + 1$ in $\C \Gamma$ consider the matrix $A_{\Phi_{\zeta}}$ defined in \eqref{eq:39}. Assume that $a_N$ is not the zero polynomial. Then $A_{\Phi_{\zeta}}$ satisfies the assumptions of the multiplicative ergodic theorem for $\mu$-almost all $\zeta \in S^1$ and we have the formula:
\begin{equation} \label{eq:44}
\log \ddet_{\Nh\Gamma} \Phi = \int_{S^1} {\nsum_j} r_j (\zeta) \chi_j (\zeta)^+ \, d\mu (\zeta) \; .
\end{equation}
Here the $\chi_j (\zeta)$ are the Ljapunov exponents of the $A_{\Phi_{\zeta}}$ with their multiplicities $r_j (\zeta)$. 
\end{theorem}

\begin{rems}
{\bf 1} Of course the theorem is valid for more general coefficient functions $a_i$.\\
{\bf 2} In case $a = a (\eta , \zeta)$ is non-zero for every $(\eta , \zeta) \in S^1 \times S^1$ formula \eqref{eq:43} also follows from the more elementary proposition \ref{t5} instead of theorem \ref{t6}.\\
{\bf 3} Applying the remark after the proof of proposition \ref{t5} we get that if $\zeta \in S^1$ is not a root of unity, the element $\Phi_{\zeta} = 1 - a (y,\zeta) x$ is a unit in $\Rh_{\zeta}$ if and only if $\int_{S^1} \log |a (\eta, \zeta)| \, d\mu (\eta) \neq 0$. Thus, if $\Phi = 1 - a (y,z) x$ is a unit in $\Nh\Gamma$ this integral is non-zero for $\mu$-almost all $\zeta$. \\
{\bf 4} In \cite{D} or by a different method in \cite{DS} it was shown that for $\Phi = \sum^N_{i=0} a_i (y,z) x^i$ in $\Z\Gamma \cap L^1 (\Gamma)^{\times} = \Z\Gamma \cap (\Nh\Gamma)^{\times}$ the number $\log \ddet_{\Nh\Gamma} \Phi$ is the entropy $h_{\Phi}$ of the natural $\Gamma$-action on the Pontrjagin dual of $\Z \Gamma / \Z \Gamma \Phi$. As explained in the introduction, the formulas for $h_{\Phi}$ resulting from theorem \ref{t11} were first found by Lind and Schmidt, c.f. \cite{LS}. 
\end{rems}

Mathematisches Institut\\
Einsteinstr. 62\\
48149 M\"unster\\
Germany
\end{document}